\documentclass[12pt, reqno]{amsart}
\usepackage{amssymb,amsthm,amsmath}
\usepackage[numbers,sort&compress]{natbib}
\usepackage{amssymb,amsmath}
\usepackage{amsfonts}
\usepackage{mathrsfs}
\usepackage{latexsym}
\usepackage{amssymb}
\usepackage{amsthm}
\usepackage{indentfirst}
\let\oldsection\section
\renewcommand\section{\setcounter{equation}{0}\oldsection}

\newtheorem{theorem}{Theorem}[section]
\newtheorem{lemma}{Lemma}[section]

\newtheorem{definition}{Definition}[section]

\newtheorem{remark}{Remark}[section]

\date{October 4, 2014}

\begin{document}

\title[Uniqueness of weak solutions to Ericksen-Leslie]{On the uniqueness of weak solutions to the Ericksen-Leslie liquid crystal model in $\mathbb R^2$}

\author{Jinkai~Li}
\address[Jinkai~Li]{Department of Computer Science and Applied Mathematics, Weizmann Institute of Science, Rehovot 76100, Israel}
\email{jklimath@gmail.com}

\author{Edriss~S.~Titi}
\address[Edriss~S.~Titi]{
Department of Mathematics, Texas A\&M University, 3368 TAMU, College Station, TX 77843-3368, USA. ALSO, Department of Computer Science and Applied Mathematics, Weizmann Institute of Science, Rehovot 76100, Israel.}
\email{titi@math.tamu.edu and edriss.titi@weizmann.ac.il}

\author{Zhouping Xin}
\address[Zhouping Xin]{The Institute of Mathematical Sciences, The Chinese University of Hong Kong, Hong Kong}
\email{zpxin@ims.cuhk.edu.hk}

\subjclass[2010]{AMS 76D03, 35D30, 76A15.}
\keywords{uniqueness; weak solutions; Ericksen-Leslie system; liquid crystals.}

\allowdisplaybreaks
\begin{abstract}This paper concerns the uniqueness of weak solutions to the Cauchy problem to the Ericksen-Leslie system of
liquid crystal models in $\mathbb R^2$, with both general Leslie stress tensors and general Oseen-Frank density. It is shown here that such a system admits a unique weak solution provided that the Frank coefficients are close to some positive constant, which solves an interesting open
problem. One of the main ideas of our proof is to perform suitable energy estimates at level one order lower than the natural
basic energy estimates for the Ericksen-Leslie system.
\end{abstract}

\maketitle

\allowdisplaybreaks

\section{Introduction}

This paper is devoted to the study of the uniqueness of weak solutions to the two-dimensional Ericksen-Leslie system modeling the flow of the nematic liquid crystals.
The Ericksen-Leslie system is one of the most successful models for the nematic liquid crystals, which is formulated by Ericksen \cite{er1} and Leslie \cite{le2} in the 1960s.

In spite of the tremendous previous progress, the existence and uniqueness of global weak solutions to the general three-dimensional Ericksen-Leslie model are still open. Up to now, it is only for the two-dimensional case that the system has been proved to have global weak solutions, see, e.g., the works by Lin--Lin--Wang \cite{LinLinWang}, Hong \cite{Hong}, Hong--Xin \cite{HongXin}, Huang--Lin--Wang \cite{HuangLinWang} and Wang--Wang \cite{WangWang}. Concerning the uniqueness of weak solutions, only some special cases have been known, see Lin--Wang \cite{LinWang}, where they proved the uniqueness of the weak solutions
to the two-dimensional system of the following form
\begin{eqnarray*}
  &&\partial_tu+(u\cdot\nabla)u-\Delta u+\nabla p=-\text{div}\,(\nabla d\odot\nabla d),\\
  &&\text{div}\,u=0,\quad|d|=1,\\
  &&\partial_t d+(u\cdot\nabla )d =\Delta d+|\nabla d|^2d,
\end{eqnarray*}
where $\nabla d\odot\nabla d$ is a $2\times2$ matrix, with the $(i,j)$-th entry $\partial_id\cdot\partial_jd$, $i, j=1,2$. Note that the above system is a simplified model of the original Ericksen-Leslie system.

In this paper, we consider the general Ericksen-Leslie system (see \cite{er1,le2}) in $\mathbb R^2$:
\begin{eqnarray}
  &&\partial_tu+(u\cdot\nabla)u+\nabla p=\text{div}\,(\sigma^E+\sigma^L),\label{MAIN1}\\
  &&\text{div}\,u=0,\quad|d|=1,\label{MAIN2}\\
  &&\partial_td+(u\cdot\nabla)d+\left(\frac{\lambda_2}{\lambda_1}A-\Omega
  \right)\cdot d=-\frac{1}{\lambda_1}(h-(d\cdot h)d)+\frac{\lambda_2}{\lambda_1}(d\cdot A\cdot d)d,\label{MAIN3}
\end{eqnarray}
where the velocity field $u=(u^1,u^2)\in\mathbb R^2$, the orientation field $d=(d^1,d^2,d^3)\in S^2$ (the unit sphere in $\mathbb R^3$) and the pressure $p\in\mathbb R$ are the unknowns, $\lambda_1$ and $\lambda_2$ are two given constants, with $\lambda_1<0$. The notations $\sigma^E, \sigma^L$ and $h$ are the Ericksen stress tensor, the Leslie stress tensor and the molecular field, respectively, whose expressions will be clear below. Here and what follows, we denote by $A$ and $\Omega$ the symmetric and skew-symmetric parts of the tensor $\nabla u$, respectively, i.e.
$$
A=\frac{1}{2}(\nabla u+\nabla u^T),\quad\Omega=\frac{1}{2}(\nabla u-\nabla u^T).
$$

We consider the following general Oseen-Frank density
$$
W(d,\nabla d)=k_1(\text{div}\,d)^2+k_2(d\cdot\text{curl}\,d)^2+k_3|d\times\text{curl}\,d|^2,
$$
with three positive constants $k_i, i=1,2,3$, which are called Frank's coefficients. As in Hong-Xin \cite{HongXin}, without loss of generality, we can suppose that
\begin{equation}\label{OF}
W(d,\nabla d)=a|\nabla d|^2+V(d,\nabla d),
\end{equation}
with $a=\min\{k_1,k_2,k_3\}$, and
\begin{equation}\label{V}
V(d,\nabla d)=(k_1-a)(\text{div}\,d)^2+(k_2-a)(d\cdot\text{curl}\,d)^2+(k_3-a) |d\times\text{curl}\,d|^2.
\end{equation}
The Ericksen stress tensor $\sigma^E$ is given by
\begin{equation}
  \sigma^E=-(\nabla d)^T\frac{\partial W(d,\nabla d)}{\partial(\nabla d)},\label{sigmaE}
\end{equation}
where $(\nabla d)^T$ is the transposed matrix of $\nabla d$.
The Leslie stress tensor $\sigma^L$ has the form
\begin{align*}
  \sigma^L=&(\mu_1\hat d\otimes\hat d:A)\hat d\otimes\hat d+\mu_2\hat N\otimes\hat d+\mu_3\hat d\otimes\hat N\\
  &+\mu_4A+\mu_5(A\cdot\hat d)\otimes\hat d+\mu_6\hat d\otimes(A\cdot\hat d),
\end{align*}
with six constants $\mu_i, i=1,2,\cdots,6$, which are called Leslie's coefficients, where
$$
\hat d=(d^1, d^2),\quad\mbox{and}\quad\hat N=\partial_t\hat d+(u\cdot\nabla)\hat d-\Omega\cdot\hat d.
$$
The molecular field $h$ in (\ref{MAIN3}) is given by
$$
h=\text{div}\left(\frac{\partial W(d,\nabla d)}{\partial(\nabla d)}\right)- \frac{\partial W(d,\nabla d)}{\partial d}.
$$
The Leslie coefficients $\mu_i, i=1,2,\cdots,6$, and the constants $\lambda_1, \lambda_2$ satisfy the relations
\begin{eqnarray}
  &\lambda_1=\mu_2-\mu_3,\quad\lambda_2=\mu_5-\mu_6,\quad\mu_2+\mu_3=\mu_6-\mu_5, \label{mus}
\end{eqnarray}
where the last equality is called Parodi's relation.

Note that $A$ and $\Omega$ are $2\times2$ matrices, as pointed out in \cite{HuangLinWang}, the meanings of $A\cdot d$ and $\Omega\cdot d$  are understood in the following way
$$
A\cdot d=(A\cdot\hat d,0),\quad\Omega\cdot d=(\Omega\cdot\hat d,0),
$$
and consequently
$$
d\cdot A\cdot d=\hat d\cdot A\cdot\hat d,\quad d\cdot\Omega\cdot d=\hat d\cdot\Omega\cdot\hat d.
$$
Such understanding is natural when supposing that the Ericksen-Leslie system depends only on two spatial variables and the flow motion is in the plane. With these notations, equation (\ref{MAIN3}) can be rewritten in the component form as
\begin{eqnarray}
  &&\partial_t\hat d+(u\cdot\nabla)\hat d+\left(\frac{\lambda_2}{\lambda_1}A-\Omega
  \right)\cdot \hat d
  =-\frac{1}{\lambda_1}(\hat h-(d\cdot h)\hat d)+\frac{\lambda_2}{\lambda_1}(\hat d\cdot A\cdot \hat d)\hat d,\label{main3-1}\\
  &&\partial_td^3+u\cdot\nabla d^3=-\frac{1}{\lambda_1}(h^3-(d\cdot h)d^3)+\frac{\lambda_2}{\lambda_1}(\hat d\cdot A\cdot \hat d)d^3,\label{main3-2}
\end{eqnarray}
where $\hat h$ is the first two components of $h$, i.e. $\hat h=(h^1, h^2)$, and all terms in (\ref{main3-1})--(\ref{main3-2}) are now understood in the usual way.


Global existence of weak solutions to the Cauchy problem of (\ref{MAIN1})--(\ref{MAIN3}) in $\mathbb R^2$ has been proved in \cite{HongXin,HuangLinWang,WangWang}, but the uniqueness was not obtained there. The aim of this paper is to prove the uniqueness of weak solutions to the Cauchy problem of system (\ref{MAIN1})--(\ref{MAIN3}), and in particular, we will show that the weak solution established in \cite{HongXin,HuangLinWang,WangWang} is unique. Since we consider the system with general Leslie stress tensor and general Oseen-Frank density, the high order coupling terms, such as $\hat d\otimes\hat N$ and $A\cdot d$, which are as high as the leading terms, appear in (\ref{MAIN1}) and (\ref{MAIN3}), the semigroup method used in \cite{LinWang} for showing the uniqueness of weak solutions for a special case of the stress tensor and Oseen-Frank density does not apply to the current general case. Actually, as it will be seen below, we will use a completely different approach to prove the uniqueness from that developed in \cite{LinWang}.

For the Cauchy problem, we complement the Ericksen-Leslie system with the following initial condition
\begin{equation}
  \label{ic}
  (u, d)|_{t=0}=(u_0, d_0),
\end{equation}
such that
\begin{equation}
  \label{ass}
  u_0\in H,\quad d_0\in H_b^1,
\end{equation}
where $H$ and $H_b^1$ are the spaces of functions defined below.

In this paper, the spaces $\mathcal D(\mathbb R^2), H, V$ and $H_b^1$ are defined as
\begin{eqnarray*}
  &&\mathcal D(\mathbb R^2)=\{\varphi\in C_0^\infty(\mathbb R^2)\,|\,\text{div}\,\varphi=0\},\\
  &&H=\mbox{the closure of }\mathcal D(\mathbb R^2)\mbox{ in }L^2(\mathbb R^2),\\
  &&V=\mbox{the closure of }\mathcal D(\mathbb R^2)\mbox{ in }H^1(\mathbb R^2),\\
  &&H_b^k=\{d\,|\,d-b\in H^k(\mathbb R^2), |d|=1\},
\end{eqnarray*}
where $b$ is a given unit constant vector. For $1\leq q\leq\infty$, $\|\cdot\|_q$ denotes the $L^q(\mathbb R^2)$ norm. For convenience, we adopt the following notation
$$
\int fdx=\int_{\mathbb R^2}fdx.
$$

The definition of weak solutions to the Ericksen-Leslie system, subject to the initial condition (\ref{ic}), is given in the following

\begin{definition}
\label{def}
Given a positive time $T\in(0,\infty)$. A couple $(u,d)$ is called a weak solution to system (\ref{MAIN1})--(\ref{MAIN3}) in $\mathbb R^2\times(0,T)$, subject to the initial condition (\ref{ic}), if the following two statements hold

(i) $u\in C([0,T]; H)\cap L^2(0,T; V)$ and $d\in C([0,T]; H_b^1)\cap L^2(0,T;H_b^2)$,

(ii) for any test vector field $\varphi\in C_0^\infty(\mathbb R^2\times[0,T))$, with $\text{div}\,\varphi=0$, it holds that
\begin{align*}
\int_0^T\int [-u\cdot\partial_t\varphi+(\sigma^E+\sigma^L-u\otimes u):\nabla\varphi]dxdt=\int u_0(x) \varphi(x,0)dx,
\end{align*}
and for any test vector field $\phi\in C_0^\infty(\mathbb R^2\times[0,T))$, it holds that
\begin{align*}
-\frac{1}{\lambda_1}&\int_0^T\int
\left(\frac{\partial W(d,\nabla d)}{\partial(\nabla d)}:\nabla(\phi-(d\cdot\phi)d)+\frac{\partial W(d,\nabla d)}{\partial d}\cdot(\phi-(d\cdot\phi)d)\right) dxdt\\
&+\int_0^T\int\left[(u\cdot\nabla)d+\left(\frac{\lambda_2}{\lambda_1}A-\Omega \right)\cdot d-\frac{\lambda_2}{\lambda_1}(d\cdot A\cdot d)d\right]\cdot\phi  dxdt\\
=&\int_0^T\int d\cdot\partial_t\phi dxdt+\int d_0(x)\phi(x,0)dx.
\end{align*}
\end{definition}

\begin{remark}
  \label{rem1}
By the definition, the weak solution $(u,d)$ has the regularity
$$
(u,\nabla d)\in L^\infty(0,T;L^2)\cap L^2(0,T;H^1),
$$
and thus, by the Ladyzhenskaya inequality, one can see that $(u,\nabla d)\in L^4(\mathbb R^2\times(0,T))$. On account of this, using equation (\ref{MAIN3}), one can further show that
$\partial_t d\in L^2(0,T; L^2).$
This, together with $d\in L^2(0,T; H^2_b)$, implies that equation (\ref{MAIN3}) is actually satisfied a.e. in $\mathbb R^2\times(0,T)$.
\end{remark}

In addition to (\ref{mus}), we assume further that $\mu_i$, $i=1,2,\cdots,6,$ satisfy
\begin{equation}
  \mu_1-\frac{\lambda_2^2}{\lambda_1}\geq0,\quad\mu_4>0,\quad\mu_5+\mu_6\geq -\frac{\lambda_2^2}{\lambda_1},\label{assiad}
\end{equation}
though some weaker assumption than this is sufficient, see Remark \ref{rem2} below.

Next, we state our main result.

\begin{theorem}
  \label{theorem}
Let $T\in(0,\infty)$ be any given positive time. Suppose that the Leslie coefficients $\mu_i, i=1,2,\cdots,6,$ satisfy (\ref{mus}) and (\ref{assiad}). Set $a=\min\{k_1,k_2,k_3\}$ and $\delta=\max\{k_1-a, k_2-a,k_3-a\}$. Then, there exists an absolute positive constant $C_0$, such that, if
$$
\delta\leq \delta_0:=\min\left\{1,\frac{|\lambda_1|}{|\lambda_1| +|\lambda_2|}\sqrt{\frac{\mu_4}{-2\lambda_1}}\right\}\frac{a}{C_0},
$$
then, for any initial data $(u_0, d_0)$ satisfying (\ref{ass}), the Ericksen-Leslie system (\ref{MAIN1})--(\ref{MAIN3}) on $\mathbb R^2\times(0,T)$, subject to the initial condition (\ref{ic}), has a unique weak solution, which depends continuously on the initial data.
\end{theorem}

As we mentioned before, the existence of weak solutions to system (\ref{MAIN1})--(\ref{MAIN3}), subject to the initial condition (\ref{ic}), has been proven in \cite{HongXin,HuangLinWang,WangWang}. Moreover, the solutions established there are smooth away from at most finite many singular times. Thus, to prove Theorem \ref{theorem}, it suffices to show the uniqueness and continuous dependence on the initial data.

A usual way for proving uniqueness of solutions is to consider the difference between two solutions and then obtain some energy estimates for the resulting system of the difference at the level of the basic natural energy of the system. One may also try to use this approach to prove the uniqueness of weak solutions to the Ericksen-Leslie system. Noticing that the basic natural energy identity for a solution $(u, d)$ to the Ericksen-Leslie system (\ref{MAIN1})--(\ref{MAIN3}) is
\begin{equation*}
  \frac{1}{2}\frac{d}{dt}\int(|u|^2+|\nabla d|^2)dx+\int(\mathcal Q(d,A,A)-\frac{1}{\lambda_1}|h-(d\cdot h)d|^2)=0, 
\end{equation*}
where
\begin{equation}\label{Q}
\mathcal Q(d,A,A)=\left(\mu_1-\frac{\lambda_2^2}{\lambda_1}\right)(d\cdot A\cdot d)^2+\mu_4|A|^2+\left(\mu_5+\mu_6+\frac{\lambda_2^2}{\lambda_1}\right)(A\cdot d)^2,
\end{equation}
one may test the momentum equation of the resulting difference system by $u$ (of course, $u$ now denotes the difference of two velocity fields) and the director equation by $-\Delta d$ ($d$ now is the difference of two director fields). Unfortunately, by following this approach, one will encounter a term like $\int|\nabla d_1|^2d\cdot\Delta ddx$, which cannot be controlled by using the embedding inequalities.

Therefore, it does not seem to be suitable to perform the energy estimates at the level of the basic natural energy of the system. Instead, to obtain the uniqueness, our strategy is to do the estimates at the level of one order lower than the basic energy. Roughly speaking, we will perform the $H^{-1}$ norm estimates on the velocity and the $L^2$ norm estimates on the director. The obvious advantage of doing so is that the basic energy automatically provides high order \textit{a priori} estimates (compared with the lower order energy). Keeping this in mind, we introduce the vector field $\xi$ associated with $u$ (the difference of two velocity fields) as
$$
\xi=(-\Delta+I)^{-1}u,
$$
and test the momentum and director equations of the difference system against $\xi$ and $d$ (the difference of two director fields),
respectively. As a result, after some careful, but standard analysis, and making full use the cancellation properties of the coupled terms between the velocity and the director fields, we can successfully derive an energy inequality, which guarantees the uniqueness of weak solutions to the Ericksen-Leslie system. Notably, a similar idea has been used by Larios--Lunasin--Titi in \cite{LarLunTiti} to prove the global well-posedness for the 2D Boussinesq system with anisotropic viscosity and without heat diffusion. Finally, we mention a relevant paper by Constantin--Sun \cite{CONSUN}, where the Lagrange approach was used to prove the uniqueness of solutions to Oldroyd-B and related complex fluid models.

\begin{remark}\label{rem2}
(i) Theorem \ref{theorem} implies that the weak solutions established in \cite{HuangLinWang} are unique. Theorem \ref{theorem} also implies that the weak solutions obtained in \cite{HongXin,WangWang} are unique, as long as the Frank coefficients $k_i, i=1,2,3,$ are close enough to each other.

(ii) The conclusion still holds if we replace (\ref{assiad}) by the following slightly weaker assumption
$$
\mu:=\min\left\{\mu_4, \mu_1+\mu_4+\mu_5+\mu_6, \mu_4+\mu_5+\mu_6+\frac{\lambda_2^2}{\lambda_1}\right\}
>0.
$$
This generalization relies on the fact that $\mathcal Q(d,A,A)$, defined by (\ref{Q}), satisfies
$$
\mathcal Q(d,A,A)\geq\mu|A|^2.
$$
Since this slightly weaker generalization is not the main task of the present paper, we will still use the assumption (\ref{assiad}) throughout this paper.
\end{remark}

At the end of this section, we give some equivalent expressions of the molecular field $h$ and the Leslie stress tensor $\sigma^L$, which will be used throughout this paper. As it will be shown below, we write $h$ and $\sigma^L$ for the general case as perturbations of those for a special case.

Due to the expression of the Oseen-Frank density function $W(d,\nabla d)$, (\ref{OF}), then straightforward calculations yield
\begin{equation}
  h=2a\Delta d+H,\label{h}
\end{equation}
with
\begin{equation}
  H=\text{div}\,\left(\frac{\partial V(d,\nabla d)}{\partial(\nabla d)}\right)-\frac{\partial V(d,\nabla d)}{\partial d}.\label{H}
\end{equation}
Since $|d|=1$, it has $\Delta d\cdot d=-|\nabla d|^2$, and as a result we have
\begin{equation}
  h-(d\cdot h)d=2a(\Delta d+|\nabla d|^2d)+H-(d\cdot H)d. \label{h-}
\end{equation}
The terms $\Delta d$ and $\Delta d+|\nabla d|^2d$ can be viewed as the main parts of $h$ and of $h-(d\cdot h)d$, respectively, while the remaining terms will be considered as the perturbations of these main parts.

It follows from (\ref{h-}), (\ref{mus}) and (\ref{main3-1}) that the Leslie stress tensor $\sigma^L$ can be rewritten as the following equivalent form
\begin{align}
  \sigma^L=&\left(\mu_1-\frac{\lambda_2^2}{\lambda_1}\right)(\hat d\cdot A\cdot\hat d)\hat d\otimes\hat d+\mu_4A\nonumber\\
  &+\left(\mu_5-\frac{\lambda_2}{\lambda_1}\mu_2\right)(A\cdot\hat d)\otimes\hat d+\left(\mu_6-\frac{\lambda_2}{\lambda_1}\mu_3\right)\hat d\otimes(A\cdot\hat d)\nonumber\\
  &-\frac{1}{\lambda_1}[\mu_2(\hat h-(d\cdot h)\hat d)\otimes\hat d+\mu_3\hat d\otimes(\hat h-(d\cdot h)\hat d)]\nonumber\\
  =:&\Sigma^L+\Pi^L,\label{E0}
\end{align}
where $\Sigma^L$ and $\Pi^L$ are given by
\begin{align}
  \Sigma^L=&\left(\mu_1-\frac{\lambda_2^2}{\lambda_1}\right)(\hat d\cdot A\cdot\hat d)\hat d\otimes\hat d+\mu_4A\nonumber\\
  &+\left(\mu_5-\frac{\lambda_2}{\lambda_1}\mu_2\right)(A\cdot\hat d)\otimes\hat d+\left(\mu_6-\frac{\lambda_2}{\lambda_1}\mu_3\right)\hat d\otimes(A\cdot\hat d)\nonumber\\
  &-\frac{2a}{\lambda_1}(\mu_2\Delta\hat d\otimes\hat d+\mu_3\hat d\otimes\Delta\hat d)-\frac{2a\lambda_2}{\lambda_1}(\Delta d\cdot d)\hat d\otimes\hat d,\label{E1}
\end{align}
where, in the last term, we have use $d\cdot\Delta d=-|\nabla d|^2$, guaranteed by $|d|=1$, and
\begin{align}
  \Pi^L=&-\frac{1}{\lambda_1}[\mu_2(\hat H-(d\cdot H)\hat d)\otimes\hat d+\mu_3\hat d\otimes(\hat H-(d\cdot H)\hat d)]\nonumber\\
  =&-\frac{1}{\lambda_1}[\mu_2\hat H\otimes\hat d+\mu_3\hat d\otimes\hat H+\lambda_2(d\cdot H)\hat d\otimes\hat d],\label{E2}
\end{align}
respectively, where $H$ is given by (\ref{H}), and $\hat H$ is the first two components of $H$, i.e. $\hat H=(H^1, H^2)$. As it will be shown in the proof of Theorem \ref{theorem} in the next section, $\Sigma^L$ will be the main part of $\sigma^L$, while $\Pi^L$ is the perturbation part.

\section{Proof of Theorem \ref{theorem}}
Let $T\in(0,\infty)$ be any given positive time. Let $(u_1, d_1)$ and $(u_2,d_2)$ be two weak solutions to system (\ref{MAIN1})--(\ref{MAIN3}) on $\mathbb R^2\times(0,T)$. Let $A_i$ and $\Omega_i$ be the symmetric and skew-symmetric parts of $\nabla u_i$, respectively. Let $H_i$, $\sigma_i^L, \Sigma^L_i$ and $\Pi_i^L$ be the corresponding quantities associated to $(u_i, d_i)$, $i=1,2$, given by (\ref{H}), (\ref{E0}), (\ref{E1}) and (\ref{E2}), respectively. Set $u=u_1-u_2$ and $d=d_1-d_2$, and let $A$ and $\Omega$ be the symmetric and skew-symmetric parts of $\nabla u$, respectively.

Define the vector fields $\xi_i=(-\Delta+I)^{-1}u_i$, in other words, for $i=1,2$, $\xi_i$ is the unique solution to
\begin{equation}
  -\Delta\xi_i+\xi_i=u_i,\quad \xi_i\rightarrow0,\mbox{ as }x\rightarrow\infty.  \label{xi}
\end{equation}
Recalling that $\text{div}\,u_i=0$, it follows that
\begin{equation}
  \label{divfree}
  \text{div}\,\xi_i=\text{div}\,(-\Delta+I)^{-1}u_i= (-\Delta+I)^{-1}\text{div}\,u_i=0.
\end{equation}

Setting $\xi=\xi_1-\xi_2$, and denote by $S$ and $Q$ the symmetric and skew-symmetric parts of $\nabla\xi$, respectively. Then, it is clear that
\begin{equation}\label{AOMEGA}
  A=-\Delta S+S, \quad\Omega=-\Delta Q+Q.
\end{equation}

Integration by parts and using the fact that $\text{div}\,\xi=0$, one can check easily that
\begin{equation}
  \int|S|^2dx=\frac{1}{2}\int|\nabla\xi|^2dx,\quad\int|\nabla S|^2dx=\frac{1}{2}\int|\nabla^2\xi|^2dx.\label{int}
\end{equation}

Before giving the proof of Theorem \ref{theorem}, we state the following two lemmas, where the proof of the first lemma, Lemma \ref{lem1}, will be given in the Appendix.

\begin{lemma}
  \label{lem1}
There exists a positive constant $C$ such that the following inequality holds
\begin{align*}
\int(\Sigma_1^L-\Sigma_2^L):\nabla\xi dx
  \geq&\frac{\mu_4}{4}\int(|\nabla\xi|^2+|\nabla^2\xi|^2)dx -\frac{2a\lambda_{2}}{\lambda_1}\int(d\cdot d_1)\hat d_1\otimes\hat d_1:\Delta Sdx\nonumber\\
  &-C\int(|\Delta d_1|+|\nabla u_2|+|\Delta d_2|)(|d|^2 +|\nabla\xi|^2)dx\nonumber\\
  &+2a\int\hat d\cdot\left(\frac{\lambda_2}{\lambda_1} \Delta S- \Delta Q\right)\cdot\hat d_1dx.
\end{align*}
\end{lemma}

\begin{lemma}
  \label{lem2}
The following estimates hold true
\begin{align*}
   \left|\int(\Pi_1^L-\Pi_2^L):\nabla\xi dx\right|
  \leq& C_1\delta\left(\left|\frac{\lambda_2}{\lambda_1}\right|+1\right)
  \int|\nabla d||\nabla^2\xi|dx+\varepsilon\int(|\nabla d|^2+|\nabla^2\xi|^2)dx\\
  &+C_\varepsilon\int(|\nabla^2d_1|+|\nabla^2d_2|)(|d|^2 +|\nabla\xi|^2)dx,
\end{align*}
and
\begin{align*}
  \left|\int(\sigma_1^E-\sigma_2^E):\nabla\xi dx\right|\leq&
  C_\varepsilon\int(|\nabla^2d_1|+|\nabla^2d_2|)(|d|^2+|\nabla\xi|^2)dx\\
  &+\varepsilon\int(|\nabla d|^2+|\nabla^2\xi|^2)dx,
\end{align*}
for any positive number $\varepsilon$, where $\delta=\max\{k_1-a, k_2-a, k_3-a\},$ and $C_1$ is an absolute positive constant.
\end{lemma}

\begin{proof}
Since $V(d,\nabla d)$ is quadratic in $d$ and $\nabla d$, then $\frac{\partial V(d,\nabla d)}{\partial d}$ and $\frac{\partial V(d,\nabla d)}{\partial(\nabla d)}$ are linear in $d$ and $\nabla d$, respectively. Thus, we have
\begin{align*}
  \left|\frac{\partial V(d_1,\nabla d_1)}{\partial d}-\frac{\partial V(d_2,\nabla d_2)}{\partial d}\right|=&\left|\frac{\partial V(d_1,\nabla d_1)}{\partial d}-\frac{\partial V(d_1,\nabla d_2)}{\partial d}+\frac{\partial V(d,\nabla d_2)}{\partial d}\right|\\
  =&\left|\int_0^1\frac{\partial^2 V(d_1,\nabla d_2+\theta\nabla d)}{\partial(\nabla d)\partial d}:\nabla d d\theta+\frac{\partial V(d,\nabla d_2)}{\partial d}\right|\\
  \leq&C_1\delta[(|\nabla d_1|+|\nabla d_2|)|\nabla d|+|\nabla d_2|^2|d|]\\
  \leq&C_1\delta[(|\nabla d_1|+|\nabla d_2|)|\nabla d|+|\nabla^2 d_2||d|],
\end{align*}
where in the last inequality we have used the fact that
$|\nabla d_2|^2=-\Delta d_2\cdot d_2$. Similarly
\begin{equation}
\left|\frac{\partial V(d_1,\nabla d_1)}{\partial \nabla d}-\frac{\partial V(d_2,\nabla d_2)}{\partial \nabla d}\right|\leq C_1\delta(|\nabla d|+|\nabla d_2||d|),\label{bq2}
\end{equation}
for some absolute positive constant $C_1$.  With the aid of these estimates, we deduce that 
\begin{align}
  &\left|\int(H_1-H_2)\cdot fdx\right|\nonumber\\
  =&\left|\int\text{div}\left(\frac{\partial V(d_1,\nabla d_1)} {\partial \nabla d}-\frac{\partial V(d_2,\nabla d_2)}{\partial \nabla d}\right)\cdot fdx\right.\nonumber\\
  &\left.-\int\left(\frac{\partial V(d_1,\nabla d_1)}{\partial \nabla d} -\frac{\partial V(d_2,\nabla d_2)}{\partial \nabla d}\right)\cdot f dx\right|\nonumber\\
  =&\left|\int\left(\frac{\partial V(d_1,\nabla d_1)}{\partial \nabla d}-\frac{\partial V(d_2,\nabla d_2)}{\partial \nabla d}\right):\nabla fdx\right.\nonumber\\
  &\left.-\int\left(\frac{\partial V(d_1,\nabla d_1)}{\partial \nabla d} -\frac{\partial V(d_2,\nabla d_2)}{\partial \nabla d}\right)\cdot f dx\right|\nonumber\\
  \leq&C_1\delta\int\{(|\nabla d|+|\nabla d_2||d|)|\nabla f|+[|\nabla^2 d_2||d|+(|\nabla d_1|+|\nabla d_2|)|\nabla d|]|f|\}dx.\label{bq1}
\end{align}
Observe that the same inequality, as above, holds true for $\hat H$.

By straightforward calculations, and using relation (\ref{mus}), one has $$
\Pi_i^L:\nabla\xi=\hat H_i\cdot\left(\frac{\lambda_2}{\lambda_1}S-Q\right)\cdot\hat d_i-\frac{\lambda_2}{\lambda_1}H_i\cdot d_i(\hat d_i\cdot S\cdot \hat d_i).
$$
Therefore, noticing that $|H_2|\leq C_1\delta(|\nabla^2d_2|+|\nabla d_2|^2)\leq C_1\delta|\nabla^2d_2|$, and using (\ref{bq1}), it follows from the Cauchy-Schwarz inequality that
\begin{align*}
  &\left|\int(\Pi_1^L-\Pi_2^L):\nabla\xi dx\right|\\
  =&\left|\int\left[(\hat H_1-\hat H_2)\cdot\left(\frac{\lambda_2}{\lambda_1}S-Q\right)\cdot\hat d_1+\hat H_2\cdot\left(\frac{\lambda_2}{\lambda_1}S-Q\right)\cdot\hat d\right]dx\right.\\
  &\left.-\frac{\lambda_2}{\lambda_1}\int\{(H_1-H_2)\cdot d_1(\hat d_1\cdot S\cdot \hat d_1)+H_2\cdot[ d_1(\hat d_1\cdot S\cdot \hat d_1)-d_2(\hat d_2\cdot S\cdot \hat d_2)]\}dx\right|\\
  \leq& C_1\delta\left(1+\left|\frac{\lambda_2}{\lambda_1}\right|\right) \int\{(|\nabla d|+|\nabla d_2||d|)(|\nabla^2\xi|+|\nabla\xi||\nabla d_1|)\\
  &+[|\nabla^2 d_2||d|+(|\nabla d_1|+|\nabla d_2|)|\nabla d|]|\nabla \xi| \}dx\\
  \leq&C_1\delta\left(\left|\frac{\lambda_2}{\lambda_1}\right|+1\right)
  \int|\nabla d||\nabla^2\xi|dx+\varepsilon\int(|\nabla d|^2+|\nabla^2\xi|^2)dx\\
  &+C_\varepsilon\int(|\nabla^2d_1|+|\nabla^2d_2|)(|d|^2 +|\nabla\xi|^2)dx,
\end{align*}
which yields the first estimate in the lemma.

Similar to (\ref{bq2}), one has that
$$
\left|\frac{\partial W(d_1,\nabla d_1)}{\partial\nabla d}- \frac{\partial W(d_2,\nabla d_2)}{\partial\nabla d}\right|\leq C(|\nabla d|+|\nabla d_2||d|),
$$
for some positive constant $C$.
Thus,  
\begin{align*}
  |\sigma_1^E-\sigma_2^E|=&\left|(\nabla d_1)^T\frac{\partial W(d_1,\nabla d_1)}{\partial\nabla d}-(\nabla d_2)^T\frac{\partial W(d_2,\nabla d_2)}{\partial\nabla d}\right|\\
  =&\left|(\nabla d)^T\frac{\partial W(d_1,\nabla d_1)}{\partial\nabla d}-(\nabla d_2)^T\left(\frac{\partial W(d_1,\nabla d_1)}{\partial\nabla d}-\frac{\partial W(d_2,\nabla d_2)}{\partial\nabla d}\right)\right|\\
  \leq&C[|\nabla d||\nabla d_1|+|\nabla d_2|(|\nabla d|+|\nabla d_2||d|)]\\
  \leq&C(|\nabla d_1|+|\nabla d_2|)|\nabla d|+C|\Delta d_1||d|,
\end{align*}
from which, by the Cauchy inequality, the second conclusion follows.
\end{proof}

With this lemma in hand, we are now ready to prove Theorem \ref{theorem}.

\begin{proof}[\textbf{Proof of Theorem \ref{theorem}}]
Let $(u_1, d_1)$ and $(u_2, d_2)$ be two weak solutions to the Ericksen-Leslie system (\ref{MAIN1})--(\ref{MAIN3}). We adopt the same notations as those stated at the beginning of this section.

Recalling that $(u_i,\nabla d_i)\in L^\infty(0,T;L^2(\mathbb R^2))\cap L^2(0,T; H^1(\mathbb R^2))$, $i=1,2$, by the Ladyzhenskaya inequality, one can obtain easily that $(u_i,\nabla d_i)\in L^4(\mathbb R^2\times(0,T))$. It follows from this and (\ref{MAIN1}) that $\partial_tu\in L^2(0,T;H^{-1})$, and consequently
\begin{eqnarray}
\xi=(-\Delta+I)^{-1}u\in L^2(0,T; H^3),\quad \partial_t\xi=(-\Delta+I)^{-1}\partial_tu\in L^2(0,T;H^1).\label{EQQ1}
\end{eqnarray}
For $i=1,2$, set
$$
F(u_i, d_i)=\sigma_i^E+\sigma^L-u_i\otimes u_i.
$$
Then it is clear that
$$
\partial_tu_i=\text{div}\,F(u_i, d_i)-\nabla p_i.
$$
Consequently, one has that
$$
\partial_t\xi=\partial_t(I-\Delta)^{-1}u=(I-\Delta)^{-1}[\text{div}\,(F(u_1, d_1)-F(u_2, d_2))-\nabla p].
$$

It follows from this equation, (\ref{EQQ1}), and the fact that $\text{div}\,\xi=0$, that
\begin{align*}
  &\frac{1}{2}\frac{d}{dt}\int(|\xi|^2+|\nabla\xi|^2)dx\\
  =&\int(\xi\cdot \partial_t\xi+\nabla\xi:\nabla\partial_t\xi)dx=\int(\xi-\Delta\xi)\cdot\partial_t\xi dx\\
  =&\int(I-\Delta)\xi\cdot (I-\Delta)^{-1}[\text{div}\,(F(u_1, d_1)-F(u_2, d_2))-\nabla p]dx\\
  =&\int \xi\cdot  [\text{div}\,(F(u_1, d_1)-F(u_2, d_2))-\nabla p]dx\\
  =&-\int(F(u_1, d_1)-F(u_2, d_2)):\nabla\xi dx,
\end{align*}
where the integration by parts has been used. This, together with the expression of $F(u_i, d_i)$ and (\ref{E0}), leads to 
\begin{align*}
  &\frac{1}{2}\frac{d}{dt}\int(|\nabla\xi|^2+|\xi|^2)dx +\int(\Sigma_1^L-\Sigma_2^L):\nabla \xi dx\nonumber\\
  =&-\int(\sigma_1^E-\sigma_2^E+\Pi_1^L-\Pi_2^L-u_1\otimes u_1+u_2\otimes u_2):\nabla \xi dx.
\end{align*}

It follows from this, Lemma \ref{lem1} and Lemma \ref{lem2} that
\begin{align}
  &\frac{1}{2}\frac{d}{dt}\int(|\nabla\xi|^2+|\xi|^2)dx+\frac{\mu_4}{4}\int (|\nabla\xi|^2+|\nabla^2\xi|^2)dx\nonumber\\
  &-\frac{2a\lambda_2}{\lambda_1}\int(d\cdot d_1)\hat d_1\otimes \hat d_1:\Delta Sdx+2a\int\hat d\cdot\left(\frac{\lambda_2}{\lambda_1}\Delta S-\Delta Q\right)\cdot\hat d_1dx\nonumber\\
  \leq&-\int(\sigma_1^E-\sigma_2^E+\Pi_1^L-\Pi_2^L-u_1\otimes u_1+u_2\otimes u_2):\nabla \xi dx\nonumber\\
  &+C\int(|\Delta d_1|+|\nabla u_2|+|\Delta d_2|)(|d|^2+|\nabla\xi|^2)dx\nonumber\\
  \leq&C_1\delta\left(\left|\frac{\lambda_2}{\lambda_1}\right|
  +1\right) \int |\nabla d||\nabla^2\xi|dx+ \int[\varepsilon(|\nabla d|^2+|\nabla^2\xi|^2)+C_\varepsilon(|\nabla^2d_1| \nonumber\\
  &+|\nabla^2d_2| +|\nabla u_2|)(|d|^2+|\nabla\xi|^2)] dx +\int (|u_1|+|u_2|)(|\xi|+|\Delta\xi|)|\nabla\xi|dx\nonumber\\
  \leq&C_1\delta\left(\left|\frac{\lambda_2}{\lambda_1}\right|
  +1\right) \int |\nabla d||\nabla^2\xi|dx+2\varepsilon \int(|\nabla d|^2+|\nabla^2\xi|^2)dx+C_\varepsilon\int(1+|u_1|^2\nonumber\\
  &+|u_2|^2+|\nabla u_2|+|\nabla^2d_1|+|\nabla^2d_2|) (|d|^2+|\xi|^2+|\nabla\xi|^2)dx, \label{estxi}
\end{align}
for any positive number $\varepsilon,$ where $C_1$ is an absolute constant.

Subtracting the equations for $d_2$ from those for $d_1$, and recalling (\ref{h-}), by tedious but simple calculations, we can see that $d$ satisfies
\begin{align}
  \partial_td&+\frac{2a}{\lambda_1}\Delta d+\left(\frac{\lambda_2}{\lambda_1}A-\Omega\right)\cdot d_1-\frac{\lambda_2}{\lambda_1}(A:\hat d_1\otimes\hat d_1)d_1\nonumber\\
  =&-\frac{1}{\lambda_1}[H_1-H_2-(d_1\cdot H_1)d_1+(d_2\cdot H_2)d_2]+g,\label{d}
\end{align}
where $g$ is given by
\begin{align*}
  g=&-(u_1\cdot\nabla d+ u\cdot\nabla d_2)-\left(\frac{\lambda_2}{\lambda_1} A_2-\Omega_2\right)\cdot d -\frac{2a}{\lambda_1}(|\nabla d_1|^2d+\nabla d:\nabla(d_1+d_2)d_2) \\
  &-\frac{\lambda_2}{\lambda_1}[(A_2:\hat d_1\otimes\hat d_1)d_1-(A_2:\hat d_2\otimes\hat d_2)d_2].
\end{align*}
The second line of the expression for $g$ can be bounded by $C|A_2||d|\leq C|\nabla u_2||d|$. Thus, $g$ can be bounded as 
\begin{equation}
  |g|\leq C(|u_1|+|\nabla d_1|+|\nabla d_2|)(|u|+|\nabla d|)+(|\nabla u_2|+|\Delta d_1|)|d|,\label{bddg}
\end{equation}
where we have used the fact that $|\nabla d_1|^2\leq |\Delta d_1|$.

Multiplying equation (\ref{d}) by $2ad$, integrating over $\mathbb R^2$, and noticing that $$
|H_2|\leq C\delta(|\nabla^2d_2|+|\nabla d_2|^2)\leq C\delta|\nabla^2d_2|,$$ then we obtain from (\ref{bq1}), (\ref{bddg}) and the Cauchy inequality that
\begin{align}
&\frac{d}{dt}\int a|d|^2dx-\frac{4a^2}{\lambda_1}\int|\nabla d|^2dx
\nonumber\\
&-\frac{2a\lambda_2}{\lambda_1}\int(d\cdot d_1)(A:\hat d_1\otimes d_1)dx+2a\int d\cdot \left(\frac{\lambda_2}{\lambda_1}A-\Omega\right)\cdot d_1dx\nonumber\\
=&-\frac{2a}{\lambda_1}\int[H_1-H_2-(d_1\cdot H_1)d_1+(d_2\cdot H_2)d_2]ddx+2a\int gddx\nonumber\\
\leq&-\frac{2a}{\lambda_1}\int[(H_1-H_2)(I-d_1\otimes d_1)\cdot d+H_2\cdot(d_1\otimes d_1-d_2\otimes d_2)\cdot d]dx\nonumber\\
&+C\int[(|u_1|+|\nabla d_1|+|\nabla d_2|)(|\Delta\xi|+|\xi|+|\nabla d|)+(|\nabla u_2|+|\Delta d_1|)|d|]|d|dx\nonumber\\
\leq&-\frac{a C_1\delta}{\lambda_1}\int\{(|\nabla d|+|\nabla d_2||d|) (|\nabla d|+|d||\nabla d_1|)+[|\nabla^2d_2||d|\nonumber\\
&+(|\nabla d_1|+|\nabla d_2|)|\nabla d|]|d|\}dx+C\int|d|^2|\nabla^2 d_2|dx+C\int[(|u_1|+|\nabla d_1|\nonumber\\
&+|\nabla d_2|)(|\Delta\xi|+|\xi|+|\nabla d|)+(|\nabla u_2|+|\Delta d_1|)|d|]|d|dx\nonumber\\
\leq&-\frac{aC_1\delta}{\lambda_1}\int |\nabla d|^2 dx +\varepsilon \int(|\nabla d|^2+|\Delta\xi|^2)dx +C_\varepsilon\int(1+|u_1|^2\nonumber\\
&+|\nabla u_2|+|\nabla^2d_1|+|\nabla^2d_2|)(|\xi|^2+|d|^2)dx, \label{estd}
\end{align}
for any positive number $\varepsilon$, where $C_1$ is an absolute constant.

Summing (\ref{estxi}) with (\ref{estd}) up, and recalling (\ref{AOMEGA}), we obtain
\begin{align*}
  &\frac{1}{2}\frac{d}{dt}\int(|\xi|^2+|\nabla\xi|^2+2a|d|^2)dx+\int\left( \frac{\mu_4}{4}(|\nabla\xi|^2+|\nabla^2\xi|^2)-\frac{4a^2}{\lambda_1}
  |\nabla d|^2\right)dx\nonumber\\
  \leq&\frac{2a\lambda_2}{\lambda_1}\int(d\cdot d_1)\hat d_1\otimes \hat d_1:Sdx-2a\int\hat d\cdot\left(\frac{\lambda_2}{\lambda_1}  S- Q\right)\cdot\hat d_1dx\nonumber\\
  &+C_1\delta\left(\left|\frac{\lambda_2}{\lambda_1}\right|
  +1\right) \int |\nabla d||\nabla^2\xi|dx-\frac{aC_1\delta}{\lambda_1}\int|\nabla d|^2dx+3\varepsilon \int(|\nabla d|^2+|\nabla^2\xi|^2)dx\nonumber\\
  &+C_\varepsilon\int(1+|u_1|^2+|u_2|^2+|\nabla u_2|+|\nabla^2d_1|+|\nabla^2d_2|) (|d|^2+|\xi|^2+|\nabla\xi|^2)dx\nonumber\\
  \leq&C_1\delta\left(\left|\frac{\lambda_2}{\lambda_1}\right|
  +1\right) \int |\nabla d||\nabla^2\xi|dx-\frac{aC_1\delta}{\lambda_1}\int|\nabla d|^2dx+3\varepsilon \int(|\nabla d|^2+|\nabla^2\xi|^2)dx\nonumber\\
  &+C_\varepsilon\int(1+|u_1|^2+|u_2|^2+|\nabla u_2|+|\nabla^2d_1|+|\nabla^2d_2|) (|d|^2+|\xi|^2+|\nabla\xi|^2)dx,
\end{align*}
for any positive number $\varepsilon$, from which, taking $\varepsilon=
\frac{1}{6}\min\left\{\frac{\mu_4}{4},-\frac{4a^2}{\lambda_1}\right\}$, we arrive at
\begin{align}
&\frac{d}{dt}\int(|\xi|^2+|\nabla\xi|^2+2a|d|^2)dx+\int\left( \frac{\mu_4}{4}(|\nabla\xi|^2+|\nabla^2\xi|^2)-\frac{4a^2}{\lambda_1}
|\nabla d|^2\right)dx\nonumber\\
\leq&C\int(1+|u_1|^2+|u_2|^2+|\nabla u_2|+|\nabla^2d_1|+|\nabla^2d_2|) (|d|^2+|\xi|^2+|\nabla\xi|^2)dx\nonumber\\
&+C_0\delta\left(\left|\frac{\lambda_2}{\lambda_1}\right|
+1\right) \int |\nabla d||\nabla^2\xi|dx-\frac{aC_0\delta}{\lambda_1}\int|\nabla d|^2dx,\label{bq3}
\end{align}
where $C_0$ is an absolute positive constant.

Set
$$
\delta_0=\min\left\{1,\frac{|\lambda_1|}{|\lambda_1| +|\lambda_2|}\sqrt{\frac{\mu_4}{-2\lambda_1}}\right\}\frac{a}{C_0},
$$
and suppose that $\delta\leq\delta_0$. Then, it follows from the Cauchy inequality that
\begin{align*}
&C_0\delta\left(\left|\frac{\lambda_2}{\lambda_1}\right|
+1\right) \int |\nabla d||\nabla^2\xi|dx-\frac{aC_0\delta}{\lambda_1}\int|\nabla d|^2dx\\
\leq&C_0\delta_0\left(\left|\frac{\lambda_2}{\lambda_1}\right|
+1\right) \int |\nabla d||\nabla^2\xi|dx-\frac{aC_0\delta_0}{\lambda_1}\int|\nabla d|^2dx\\
\leq&a\sqrt{\frac{\mu_4}{-2\lambda_1}}\int|\nabla d||\nabla^2\xi|dx-\frac{a^2}{\lambda_1}\int|\nabla d|^2dx\\
=&\int\sqrt{\frac{\mu_4}{4}}|\nabla^2\xi|\sqrt{\frac{2}{-\lambda_1}} a|\nabla d|dx-\frac{a^2}{\lambda_1}\int|\nabla d|^2dx\\
\leq&\frac{1}{2}\int\left(\frac{\mu_4}{4}|\nabla^2\xi|^2-\frac{2a^2}{\lambda_1} |\nabla d|^2\right)dx-\frac{a^2}{\lambda_1}\int|\nabla d|^2dx\\
=&\int\left(\frac{\mu_4}{8}|\nabla^2\xi|^2-\frac{2a^2}{\lambda_1} |\nabla d|^2\right)dx,
\end{align*}
and consequently, by (\ref{bq3}), we have
\begin{align*}
&\frac{d}{dt}\int(|\xi|^2+|\nabla\xi|^2+2a|d|^2)dx+\int\left( \frac{\mu_4}{8}(|\nabla\xi|^2+|\nabla^2\xi|^2)-\frac{2a^2}{\lambda_1}
|\nabla d|^2\right)dx\nonumber\\
\leq&C\int(1+|u_1|^2+|u_2|^2+|\nabla u_2|+|\nabla^2d_1|+|\nabla^2d_2|) (|d|^2+|\xi|^2+|\nabla\xi|^2)dx:=J,
\end{align*}

By the H\"older, Ladyzhenskaya and Cauchy inequalities, we can estimate $J$ as
\begin{align*}
  J\leq&C(\|u_1\|_4^2+\|u_2\|_4^2+\|\nabla u_2\|_2+\|\Delta d_1\|_2+\|\Delta d_2\|_2)(\|d\|_4+\|\xi\|_4+\|\nabla\xi\|_4)^2\\
  &+C(\|d\|_2^2+\|\xi\|_2^2+\|\nabla\xi\|_2^2)\\
  \leq&C(\|u_1\|_4^2+\|u_2\|_4^2+\|\nabla u_2\|_2+\|\Delta d_1\|_2+\|\Delta d_2\|_2)(\|d\|_2+\|\xi\|_2+\|\nabla\xi\|_2)\\
  &\times(\|\nabla d\|_2+\|\nabla \xi\|_2+\|\nabla^2\xi\|_2)+C(\|d\|_2^2+\|\xi\|_2^2+\|\nabla\xi\|_2^2)\\
  \leq&\varepsilon(\|\nabla d\|_2^2+\|\nabla \xi\|_2^2+\|\nabla^2\xi\|_2^2)+C(1+\|u_1\|_4^4+\|u_2\|_4^4+\|\nabla u_2\|_2^2\\
  &+\|\Delta d_1\|_2^2+\|\Delta d_2\|_2^2)(\|d\|_2^2+\|\xi\|_2^2+\|\nabla\xi\|_2^2),
\end{align*}
with $\varepsilon= \min\left\{\frac{\mu_4}{16},-\frac{a^2}{\lambda_1}\right\}$. Therefore, we obtain that
\begin{align*}
\frac{d}{dt}(2a\|d\|_2^2+\|\xi\|_2^2+\|\nabla\xi\|_2^2)&+c_0(\|\nabla\xi\|_2^2+ \|\nabla^2\xi\|_2^2+\|\nabla d\|_2^2)\\
\leq& Cm(t)(\|d\|_2^2+\|\xi\|_2^2+\|\nabla\xi\|_2^2),
\end{align*}
where $c_0=\min\left\{\frac{\mu_4}{16},-\frac{a^2}{\lambda_1}\right\}$ and
$$
m(t)=(1+\|u_1\|_4^4+\|u_2\|_4^4+\|\nabla u_2\|_2^2
+\|\Delta d_1\|_2^2+\|\Delta d_2\|_2^2)(t).
$$
Recall that $u_i,\nabla d_i\in L^\infty(0,T;L^2)\cap L^2(0,T;H^1)$, which, by the Ladyzhenskaya inequality, implies $u_i\in L^4(0,T; L^4)$, and thus $m\in L^1((0,T))$. Then, the continuous dependence on the initial data follows from the above inequality by the Gronwall's.
\end{proof}

\section{Appendix: proof of Lemma \ref{lem1}}

Before proving Lemma \ref{lem1}, we introduce some notations. For arbitrary $(u,d)$, denote by $A$, as before, the symmetric part of $\nabla u$, and by $\hat d=(d^1, d^2)$. Let $\sigma^L$ be the Leslie stress tensor associated to $(u,d)$. Recalling the expression for $\Sigma^L$ from (\ref{E1}), we can decompose it as
\begin{equation}
\Sigma^L=\mathscr A+\mathscr B+\mathscr C, \label{sigmaLABC}
\end{equation}
where $\mathscr A, \mathscr B$ and $\mathscr C$ are given by
\begin{eqnarray}
  \mathscr A&=&\left(\mu_1-\frac{\lambda_2^2}{\lambda_1}\right)(\hat d\cdot A\cdot\hat d)\hat d\otimes\hat d+\mu_4A\nonumber\\
  &&+\left(\mu_5-\frac{\lambda_2}{\lambda_1}\mu_2\right)(A\cdot\hat d)\otimes\hat d+\left(\mu_6-\frac{\lambda_2}{\lambda_1}\mu_3\right)\hat d\otimes(A\cdot\hat d),\label{A}\\
  \mathscr B&=&-\frac{2a}{\lambda_1}(\mu_2\Delta\hat d\otimes\hat d+\mu_3\hat d\otimes\Delta\hat d),\qquad\mathscr C=-\frac{2a\lambda_2}{\lambda_1}(\Delta d\cdot d)\hat d\otimes\hat d.\label{BC}
\end{eqnarray}

For any $2\times2$ matrix $M$, denote by $M_s$ and $M_a$ the symmetric and skew-symmetric parts of $M$, respectively, that is
$$
M_s=\frac{1}{2}(M+M^T),\quad M_a=\frac{1}{2}(M-M^T).
$$
It follows from(\ref{mus}) that
\begin{eqnarray*}
  \mathscr A:M&=&\left(\mu_1-\frac{\lambda_2^2}{\lambda_1}\right)(\hat d\cdot A\cdot\hat d)\hat d\otimes\hat d:M_s+\mu_4A:M_s\\
  &&+\left(\mu_5+\mu_6-\frac{\lambda_2}{\lambda_1}(\mu_2+\mu_3)\right)(A\cdot\hat d)\otimes\hat d:M_s\\
  &&+\left(\mu_5-\mu_6-\frac{\lambda_2}{\lambda_1}(\mu_2-\mu_3)\right) (A\cdot\hat d)\otimes \hat d:M_a\\
  &=&\alpha(\hat d\cdot A\cdot\hat d)\hat d\otimes\hat d:M_s+\mu_4A:M_s+\beta(A\cdot\hat d)\otimes\hat d:M_s\\
  &=&\alpha\mathscr H:M_s+\beta\mathscr M:M_s+\mu_4A:M_s,
\end{eqnarray*}
and
\begin{eqnarray*}
  \mathscr B:M&=&-\frac{2a}{\lambda_1}[(\mu_2+\mu_3)\Delta\hat d\otimes\hat d:M_s+(\mu_2-\mu_3)\Delta\hat d\otimes\hat d:M_a]\\
  &=&\frac{2a\lambda_2}{\lambda_1}\Delta\hat d\otimes\hat d:M_s-2a\Delta\hat d\otimes\hat d:M_a\\
  &=&2a\Delta\hat d\otimes\hat d:\left(\frac{\lambda_2}{\lambda_1}M_s-M_a\right),
\end{eqnarray*}
where the constants $\alpha$ and $\beta$, and matrices $\mathscr H$ and $\mathscr M$ are given by
\begin{eqnarray*}
  &&\alpha=\mu_1-\frac{\lambda_2^2}{\lambda_1}, \quad\beta=\mu_5+\mu_6+\frac{\lambda_2^2}{\lambda_1},\nonumber\\
  &&\mathscr H=(\hat d\cdot A\cdot\hat d)\hat d\otimes\hat d,\quad\mathscr M= (A\cdot\hat d)\otimes\hat d.\label{HM}
\end{eqnarray*}
Consequently, (\ref{sigmaLABC}) yields 
\begin{eqnarray}
  \Sigma^L:M&=&(\mathscr A+\mathscr B+\mathscr C):M\nonumber\\
  &=&(\alpha\mathscr H+\beta\mathscr M+\mu_4A+\mathscr C):M_s+2a\Delta\hat d\otimes\hat d: \left(\frac{\lambda_2}{\lambda_1}M_s-M_a\right),\label{sigmaLM}
\end{eqnarray}
for any $2\times2$ matrix $M$.

We are now ready to prove Lemma \ref{lem1}.

\begin{proof}[\textbf{Proof of Lemma \ref{lem1}}]
Let $(u_1, d_1)$ and $(u_2, d_2)$ be two weak solutions to the Ericksen-Leslie system (\ref{MAIN1})--(\ref{MAIN3}), and set $u=u_1-u_2$ and $d=d_1-d_2$. We adopt the same notations introduced above, and add the subscript to distinct them (of course, no subscript is required for those associated to $(u,d)$).

We have the following
\begin{eqnarray}
  \mathscr H_1-\mathscr H_2&=&(\hat d_1\cdot A_1\cdot\hat d_1)\hat d_1\otimes\hat d_1-(\hat d_2\cdot A_2\cdot\hat d_2)\hat d_2\otimes\hat d_2
  \nonumber\\
  &=&(A_1:\hat d_1\otimes\hat d_1)\hat d_1\otimes\hat d_1-(A_2:\hat d_2\otimes\hat d_2)\hat d_2\otimes\hat d_2\nonumber\\
  &=&(A:\hat d_1\otimes\hat d_1)\hat d_1\otimes\hat d_1+ {\mathscr H}_r,\label{n1}\\
  \mathscr M_1-\mathscr M_2&=&(A_1\cdot\hat d_1)\otimes\hat d_1-(A_2\cdot\hat d_2)\otimes\hat d_2=(A\cdot\hat d_1)\otimes\hat d_1+{\mathscr M}_r,\label{n2}\\
  \mathscr C_1-\mathscr C_2&=&-\frac{2a\lambda_2}{\lambda_1}[(\Delta d_1\cdot d_1)\hat d_1\otimes\hat d_1-(\Delta d_2\cdot d_2)\hat d_2\otimes\hat d_2]\nonumber\\
  &=&-\frac{2a\lambda_2}{\lambda_1}(\Delta d\cdot d_1)\hat d_1\otimes\hat d_1+ {\mathscr C}_r,\label{n3}
\end{eqnarray}
where $ {\mathscr H}_r, {\mathscr M}_r$ and ${\mathscr C}_r$ are given by
\begin{eqnarray*}
 {\mathscr H}_r&=&(A_2:\hat d_1\otimes\hat d_1)\hat d_1\otimes\hat d_1-(A_2:\hat d_2\otimes\hat d_2)\hat d_2\otimes\hat d_2,\label{HR}\\
 {\mathscr M}_r&=&(A_2\cdot\hat d_1)\otimes\hat d_1-(A_2\cdot\hat d_2)\otimes\hat d_2,\label{MR}\\
 {\mathscr C}_r&=&-\frac{2a\lambda_2}{\lambda_1}[(\Delta d_2\cdot d_1)\hat d_1\otimes\hat d_1-(\Delta d_2\cdot d_2)\hat d_2\otimes\hat d_2].\label{CR}
\end{eqnarray*}

Since $|d_i|=1$, one can check easily that
\begin{equation}
  |{\mathscr H}_r|\leq4|A_2||\hat d|,\quad|{\mathscr M}_r|\leq2|A_2||\hat d|,\label{hm}
\end{equation}
and
\begin{equation}
  |{\mathscr C}_r|\leq2a\left|\frac{\lambda_2}{\lambda_1}\right||\Delta d_2|(|d|+2|\hat d|)\leq6a\left|\frac{\lambda_2}{\lambda_1}\right||\Delta d_2||d|.\label{c}
\end{equation}

Recall that $S$ and $Q$ are the symmetric and skew-symmetric parts of $\nabla\xi$, respectively. Thus, it follows from (\ref{sigmaLM}), and (\ref{n1})--(\ref{n3}) that
\begin{align*}
  &\int(\Sigma_1^L-\Sigma_2^L):\nabla\xi dx\\
  =&\int(\alpha(\mathscr H_1-\mathscr H_2)+\beta(\mathscr M_1-\mathscr M_2)+\mu_4A+\mathscr C_1-\mathscr C_2):Sdx\\
  &+2a\int(\Delta\hat d_1\otimes\hat d_1-\Delta\hat d_2\otimes\hat d_2)
  :\left(\frac{\lambda_2}{\lambda_1}S-Q\right)dx\\
  =&\int[\alpha(A:\hat d_1\otimes\hat d_1)\hat d_1\otimes\hat d_1+\beta(A\cdot\hat d_1)\otimes\hat d_1+\mu_4A]:Sdx\nonumber\\
   &-\frac{2a\lambda_2}{\lambda_1}\int(\Delta d\cdot d_1)\hat d_1\otimes\hat d_1:Sdx+2a\int(\Delta\hat d\otimes\hat d_1):\left(\frac{\lambda_2}{\lambda_1}S-Q\right)dx\nonumber\\
   &+\int\left[2a(\Delta\hat d_2\otimes\hat d):\left(\frac{\lambda_2}{\lambda_1}S-Q\right)+(\alpha {\mathscr H}_r+\beta {\mathscr M}_r+ {\mathscr C}_r):S\right]dx\\
   =:&I_1+I_2+I_3+I_4.
\end{align*}

Next, we estimate the terms $I_i, i=1,2,3,4$, respectively.
Noticing that $\alpha,\beta\geq0$, and that $A=-\Delta S+S$; integration by parts yields
\begin{align*}
  I_1=&-\int[\alpha(\Delta S:\hat d_1\otimes\hat d_1)\hat d_1\otimes\hat d_1+\beta(\Delta S\cdot\hat d_1)\otimes\hat d_1+\mu_4\Delta S]:Sdx\\
  &+\int[\alpha(S:\hat d_1\otimes\hat d_1)\hat d_1\otimes\hat d_1+\beta(S\cdot\hat d_1)\otimes\hat d_1+\mu_4S]:Sdx\\
  =&\int[\alpha(\partial_k S:\hat d_1\otimes\hat d_1)\hat d_1\otimes\hat d_1+\beta(\partial_k S\cdot\hat d_1)\otimes\hat d_1+\mu_4\partial_k S]:\partial_kSdx\\
  &+\int[\alpha(\partial_k S:\partial_k (\hat d_1\otimes\hat d_1))\hat d_1\otimes\hat d_1+
  \alpha(\partial_k S:\hat d_1\otimes\hat d_1)\partial_k (\hat d_1\otimes\hat d_1)\\
  &+\beta(\partial_k S\cdot\partial_k\hat d_1)\otimes\hat d_1
  +\beta(\partial_k S\cdot\hat d_1)\otimes\partial_k\hat d_1]:Sdx\\
  &+\int[\alpha(S:\hat d_1\otimes\hat d_1)\hat d_1\otimes\hat d_1+\beta(S\cdot\hat d_1)\otimes\hat d_1+\mu_4S]:Sdx\\
  \geq&\mu_4\int(|S|^2+|\nabla S|^2)dx-(4\alpha+2\beta)\int|\nabla S||\nabla\hat d_1||S|dx,
\end{align*}
and thus, by (\ref{int}), and noticing that $|S|\leq|\nabla\xi|$ and $|\nabla S|\leq|\nabla^2\xi|$, we arrive at
$$
I_1\geq \frac{\mu_4}{2}\int(|\nabla\xi|^2+|\nabla^2\xi|^2)dx
  -C\int|\nabla \xi||\nabla^2\xi||\nabla d_1|dx.
$$
Integration by parts, and noticing that $|\nabla d_1|^2\leq|\Delta d_1|$, one gets that 
\begin{align*}
  I_2=&-\frac{2a\lambda_2}{\lambda_1}\int(\Delta d\cdot d_1)\hat d_1\otimes\hat d_1:Sdx\\
  =&-\frac{2a\lambda_2}{\lambda_1}\int\big\{(d\cdot d_1)[\hat d_1\otimes\hat d_1:\Delta S+2\partial_k(\hat d_1\otimes\hat d_1):\partial_kS+\Delta(\hat d_1\otimes\hat d_1):S]\\
  &+2(d\cdot\partial_kd_1)\partial_k(\hat d_1\otimes\hat d_1:S)+(d\cdot\Delta d_1)\hat d_1\otimes\hat d_1:S\big\}dx\\
  \geq&-\frac{2a\lambda_2}{\lambda_1}\int(d\cdot d_1)\hat d_1\otimes\hat d_1:\Delta Sdx-C\int(|\nabla d_1||\nabla S|+|\Delta d_1||S|)|d|dx\\
  \geq&-\frac{2a\lambda_2}{\lambda_1}\int(d\cdot d_1)\hat d_1\otimes\hat d_1:\Delta Sdx-C\int(|\nabla d_1||\nabla^2\xi|+|\Delta d_1||\nabla\xi|)|d|dx.
\end{align*}
For $I_3$, integration by parts yields
\begin{align*}
  I_3=&2a\int\Delta\hat d\otimes\hat d_1:\left(\frac{\lambda_2}{\lambda_1}S-Q\right)dx=2a\int\Delta\hat d\cdot\left(\frac{\lambda_2}{\lambda_1}S-Q\right)\cdot\hat d_1dx\\
  =&2a\int\left[\hat d\cdot\left(\frac{\lambda_2}{\lambda_1}\Delta S-\Delta Q\right)
  \cdot\hat d_1+2\hat d\cdot\partial_k\left(\frac{\lambda_2}{\lambda_1}S-Q\right)\cdot\partial_k\hat d_1\right.\\
  &\left.+\hat d\cdot\left(\frac{\lambda_2}{\lambda_1}S-Q\right)\cdot\Delta\hat d_1\right]dx\\
  \geq&2a\int\hat d\cdot\left(\frac{\lambda_2}{\lambda_1}\Delta S-\Delta Q\right)
  \cdot\hat d_1dx-C\int|d|(|\nabla d_1||\nabla^2\xi|+|\nabla\xi||\Delta d_1|)dx.
\end{align*}
Finally, due to (\ref{hm}) and (\ref{c}), $I_4$ can be estimated as
$$
I_4\geq-C\int(|\nabla u_2|+|\Delta d_2|)|d||\nabla\xi|dx.
$$

On account of the estimates for $I_1,I_2,I_3,I_4$, by the Cauchy inequality, and using $|\nabla d_1|^2\leq|\Delta d_1|$, we deduce
\begin{align*}
  &\int(\Sigma_1^L-\Sigma_2^L):\nabla\xi dx=I_1+I_2+I_3+I_4\nonumber\\
  \geq&\frac{\mu_4}{2}\int(|\nabla\xi|^2+|\nabla^2\xi|^2)dx -\frac{2a\lambda_{2}}{\lambda_1}\int(d\cdot d_1)\hat d_1\otimes\hat d_1:\Delta Sdx\nonumber\\
  &+2a\int\hat d\cdot\left(\frac{\lambda_2}{\lambda_1}\Delta S-\Delta Q\right)\cdot\hat d_1dx-C\int[|\nabla\xi||\nabla^2\xi||\nabla d_1|\nonumber\\
  &+|\nabla d_1||\nabla^2\xi||d|+|\Delta d_1||\nabla\xi||d|+(|\nabla u_2|+|\Delta d_2|)|d||\nabla\xi|] dx \nonumber\\
  \geq&\frac{\mu_4}{4}\int(|\nabla\xi|^2+|\nabla^2\xi|^2)dx -\frac{2a\lambda_{2}}{\lambda_1}\int(d\cdot d_1)\hat d_1\otimes\hat d_1:\Delta Sdx\nonumber\\
  &+2a\int\hat d\cdot\left(\frac{\lambda_2}{\lambda_1}\Delta S-\Delta Q\right)\cdot\hat d_1dx-C\int(|\Delta d_1|+|\nabla u_2|\nonumber\\
  &+|\Delta d_2|)(|d|^2 +|\nabla\xi|^2)dx,
\end{align*}
which proves Lemma \ref{lem1}.
\end{proof}

\section*{Acknowledgments}
{E.S.T. is thankful to the kind hospitality of the Institute for Pure and Applied Mathematics (IPAM), where part of this work was completed.
 The work of E.S.T.~is supported in part by the NSF
grants DMS-1109640 and DMS-1109645. The work of Z.X.~is supported
partially by Zheng Ge Ru Foundation, Hong Kong RGC Earmarked
Research Grants CUHK4041/11P and CUHK4048/13P, a Focus Area Grant from The
Chinese University of Hong Kong, and a grant from Croucher Foundation.}
\par

\end{document}